 \input epsf.sty
 \input amssym.def
 \documentclass{article}
 \usepackage{amsthm}
 \def\fixspacing{%
	\oddsidemargin=0.25truein
	\evensidemargin=\oddsidemargin
	\textwidth=6.0truein
	\textheight=8.5truein
	\topmargin=-0.25truein
 }%
 \fixspacing
\long\def\proclaim#1#2\endproclaim{%
	\medbreak\noindent{\bf #1.\enspace}{\sl #2}\par\medbreak}
\long\def\mildproclaim#1:{\noindent{\em #1:\/}}
\def\Cal#1{{\cal #1}}

\long\def\explain#1\endexplain{{\sl #1}}

\def\C{\Cal{C}}
\let\dC\C
\def\G{\Cal{G}}
\def\D{\Cal{D}}
\let\Ga\Gamma
\let\De\Delta
\def\F{{\Cal F}}
\let\Si\Sigma
\def\X{{\Cal X}}
\let\bdy\partial
\let\clo\overline
\let\ori\omega

\let\ab\allowbreak
\def\mR{{\Bbb R}}
\def\itj#1{{\em\frenchspacing #1\/}}
\def\bfj#1{{\bf\frenchspacing #1\/}}

\def\nos#1{N_{{#1}}}
\def\cover#1{\widetilde{#1}}
\theoremstyle{plain}
\newtheorem{theorem}{Theorem}[section]
\newtheorem{lemma}[theorem]{Lemma}
\newtheorem{corollary}[theorem]{Corollary}
\newtheorem{observation}[theorem]{Observation}
\newtheorem{conjecture}[theorem]{Conjecture}
\newtheorem{proposition}[theorem]{Proposition}
\newtheorem*{theoremMainOSE}{Theorem \ref{MainOSE}}
\theoremstyle{definition}

\newtheorem{operation}[theorem]{Operation}
\newcounter{clm}
\def\startclaim{\setcounter{clm}{0}}
\newenvironment{claim}%
  {\refstepcounter{clm}\smallbreak\noindent{\sl Claim \theclm.\ }}%
  {\smallbreak}%
\def\claimproofblock{\vrule depth0pt height7pt width0.5pt%
   \kern-0.1pt%
   \vrule depth-6.5pt height7pt width4.0pt%
   \kern-4.0pt%
   \vrule depth0pt height0.5pt width4.0pt%
   \kern-0.1pt%
   \vrule depth0pt height7pt width0.5pt}
\newenvironment{claimproof}%
  {\smallbreak\noindent{\em Proof of Claim.\ }\ignorespaces}%
  {\kern0.5em\claimproofblock\smallbreak}%
 \title{Orientable embeddings and orientable cycle double covers of
	projective-planar graphs%
	\thanks{The United States Government is authorized to reproduce
	and distribute reprints notwithstanding any copyright notation
	herein.}%
}
 \author{%
	M. N. Ellingham%
	\thanks{Supported by U.S. National Security Agency grants
		H98230-06-1-0051 and H98230-09-1-0065.}\\%
	{\sl Department of Mathematics, 1326 Stevenson Center}\\%
	{\sl Vanderbilt University, Nashville, TN 37240, U. S. A.}\\%
	{\tt mark.ellingham@vanderbilt.edu}\\%
 \and%
	Xiaoya Zha%
	\thanks{Supported by U.S. National Security Agency grants
		H98230-06-1-0045 and H98230-09-1-0041.}\\%
	{\sl Department of Mathematical Sciences}\\%
	{\sl Middle Tennessee State University, Murfreesboro, TN 37132,
		U. S. A.}\\%
	{\tt xzha@mtsu.edu}\\%
 }
 \date{13 November 2009}

\begin{document}
 \maketitle
 \begin{abstract}
 In a {\em closed $2$-cell embedding\/} of a graph each face is
homeomorphic to an open disk and is bounded by a cycle in the graph. 
The Orientable Strong Embedding Conjecture says that every $2$-connected
graph has a closed $2$-cell embedding in some orientable surface.  This
implies both the Cycle Double Cover Conjecture and the Strong Embedding
Conjecture.  
 In this paper we prove that every $2$-connected projective-planar cubic
graph has a closed $2$-cell embedding in some orientable surface.  The
three main ingredients of the proof are (1) a surgical method to convert
nonorientable embeddings into orientable embeddings; (2) a reduction for
$4$-cycles for orientable closed $2$-cell embeddings, or orientable
cycle double covers, of cubic graphs; and (3) a structural result for
projective-planar embeddings of cubic graphs.
 We deduce that every $2$-edge-connected projective-planar graph (not
necessarily cubic) has an orientable cycle double cover.
 \end{abstract}

 \section{Introduction}\label{Intro}

 In this paper all graphs are finite and may have multiple edges but no
loops.  A graph is {\em simple\/} if it has no multiple edges.
 A {\em pseudograph\/} may have multiple edges and loops.
 By a {\em surface\/} we mean a connected compact $2$-manifold without
boundary.  The nonorientable surface of genus $k$ is denoted $\nos{k}$.
By an {\em open\/} or {\em closed disk\/} in a surface we mean a subset
of the surface homeomorphic to such a subset of $\mR^2$.

 A {\em closed 2-cell embedding\/} of a graph is an embedding such that
every face is an open disk bounded by a cycle (no repeated vertices) in
the graph.  A graph must be $2$-connected to have a closed $2$-cell
embedding.
 The {\em Strong Embedding Conjecture\/} due to Haggard \cite{Ha} says
that every $2$-connected graph has a closed $2$-cell embedding in some
surface. The even stronger {\em Orientable Strong Embedding
Conjecture\/} \cite{Ja} says that every $2$-connected graph has a closed
$2$-cell embedding in some orientable surface.
 The facial walks of every closed $2$-cell embedding form a {\em cycle
double cover\/}, a set of cycles in the graph such that each edge is
contained in exactly two of these cycles.  Therefore, both embedding
conjectures imply the well-known {\em Cycle Double Cover Conjecture\/}, 
which says that every $2$-edge-connected  graph has a cycle double
cover, 


 Every spherical embedding of a $2$-connected planar graph is an
orientable closed $2$-cell embedding.
 For cubic graphs (but not in general) we can go from a cycle double
cover back to a closed $2$-cell embedding; thus,
cubic graphs with cycle double covers (see for example
\cite{AlGoZh, GoDiss, HaggMark}) have closed $2$-cell embeddings.
 Some special classes of graphs are known to have minimum genus
embeddings with all faces bounded by cycles; these are closed $2$-cell
embeddings.  For example, the complete graph $K_n$ has embeddings with
all faces bounded by $3$-cycles in a nonorientable surface if $n
\equiv 0$, $1$, $3$ or $4$ mod $6$, and in an orientable surface if
$n \equiv 0$, $3$, $4$ or $7$ mod $12$ (see \cite{RiMCT}).
 The complete bipartite graph $K_{m,n}$ has embeddings with all faces
bounded by $4$-cycles in a nonorientable surface if $mn$ is even, and in
an orientable surface if $(m-2)(n-2)$ is divisible by $4$ \cite{Ri65n,
Ri65o}.
 However, in general not many graphs are known to have closed $2$-cell
embeddings.

 Even though the Orientable Strong Embedding Conjecture is very strong,
the study of orientable closed $2$-cell embeddings seems to be
promising.  One approach is to try to prove the following:

 \begin{conjecture}[Robertson and Zha {\rm [personal communication]}]\label{NSEtoOSE}
 If a $2$-connected graph has a nonorientable closed $2$-cell embedding 
then it has an orientable closed $2$-cell embedding.
 \end{conjecture}

 \noindent
 While this conjecture appears weaker than the Orientable Strong
Embedding Conjecture, it is actually equivalent to it, via a result in
\cite{Z} based on techniques of Little and Ringeisen \cite{LR}.  This
result says that if a graph $G$ has an orientable closed $2$-cell
embedding, and $e$ is a new edge, then $G + e$  has a closed $2$-cell
embedding in {\em some\/} surface, which may or may not be orientable.

 This paper is a first step towards verifying Conjecture \ref{NSEtoOSE}.
Our goal is to develop techniques to turn nonorientable closed $2$-cell
embeddings into orientable closed $2$-cell embeddings.  Using techniques
of this kind, we show that the Orientable Strong Embedding Conjecture is
true for projective-planar cubic graphs. The following is the main
result of this paper.
 
 \begin{theorem}\label{MainOSE}
 \global\def\MainOSEtext{%
 Every $2$-connected projective-planar cubic graph has a closed
$2$-cell embedding in some orientable surface.
 }\MainOSEtext
 \end{theorem}

 A standard construction then provides a result for general
$2$-edge-connected projective-planar graphs.
 Given such an embedded graph $G$, we may construct a $2$-edge-connected
(hence $2$-connected) cubic projective-planar graph $H$ by expanding
each vertex $v$ of degree $d_v \ge 4$ to a contractible cycle of length
$d_v$.  By Theorem \ref{MainOSE}, $H$ has an orientable closed $2$-cell
embedding, whose oriented faces give a compatibly oriented (each edge
occurs once in each direction) cycle double cover $\dC$ of $H$. 
Contracting the new cycles to recover $G$, $\dC$ becomes a compatibly
oriented cycle double cover $\dC'$ of $G$.  ($\dC'$ may not correspond to
a surface embedding of $G$, but that does not matter.)
 Some cycles of $\dC$ may become oriented eulerian subgraphs, rather than
cycles, of $\dC'$, but those can always be decomposed into oriented
cycles, preserving the compatible orientation.
  This proves the following.

 \begin{corollary}\label{MainOCDC}
 Every $2$-edge-connected projective-planar graph has an orientable
cycle double cover.
 \end{corollary}

 In Section \ref{DefNot} we introduce some notation. In
Section \ref{nonori2ori} we develop surgeries that convert nonorientable
surfaces into orientable surfaces, and use them to prove a special case
of our main result.  In Section \ref{Redn} we describe some reductions
for our problem.  In Section \ref{Main} we arrive at our main result by
proving a structural result which shows that either the special
case from Section \ref{nonori2ori} occurs, or some kind of
reduction applies.

\section{Definitions and notation}\label{DefNot}

 Let $\Psi$ denote an embedding of a graph $G$ in a surface $\Si$.
 We usually identify the graph and the point-set of its image under the
embedding.
 If $S \subseteq \Si$, then $\clo S$ denotes the closure of $S$ in
$\Si$.
 A {\em face\/} is a component of $\Si - G$.  The boundary of the face
$f$ is denoted by $\bdy f$.  Each component of $\bdy f$ is traced out by
a closed walk in $G$, which we call a {\em facial boundary component
walk\/} of $f$.  A {\em $k$-cycle face\/} is a face with exactly one
boundary component, which is a $k$-cycle.

 An embedding in which every face is an open disk ($2$-cell) is an {\em
open $2$-cell embedding\/}; then each face has a single boundary
component walk, called the {\em facial walk\/}.
 If every facial walk of an open $2$-cell embedding is in fact a
cycle, we have a {\em closed $2$-cell embedding\/}.  
 If $\Si$ is not the sphere, then the {\em representativity\/} of any
embedding $\Psi$ is defined to be
 $\rho (\Psi ) = \min \{|\Ga \cap G |:
	\Ga$ is a noncontractible simple closed	curve in $\Si \}$.
 We say $\Psi$ is {\em $k$-representative\/} if $\rho(\Psi) \ge k$.
 Robertson and Vitray \cite{RV} showed that $\Psi$ is open $2$-cell
exactly when $G$ is connected and $\Psi$ is $1$-representative, and
closed $2$-cell exactly when $G$ is $2$-connected and $\Psi$ is
$2$-representative.

 Suppose $C$ is a cycle with a given orientation, and $u$ and $v$ are
two vertices on  $C$. Denote by $u C v$ the path on $C$ from $u$ to $v$
in the given direction.
 If we have a graph embedded on an orientable surface, then all cycles
may be oriented in a consistent clockwise direction, and we assume this
orientation unless otherwise specified.
 If $P$ is a path, $uPv$ is defined to be the subpath from $u$ to $v$
along $P$.

 We say two sets (usually faces) $f$ and $g$ {\em touch\/} if $\bdy f
\cap \bdy g \ne \emptyset$; we say they {\em touch $k$ times\/} if $\bdy
f \cap \bdy g$ has $k$ components.
 A sequence of sets $f_0, f_1, f_2, \ldots, f_n$ is an {\em $f_0
f_n$-face chain of length $n$\/} if for $1 \le i \le n-1$ each $f_i$ is
a distinct face of $\Psi$ and for $0 \le i < j \le n$, $f_i$ and $f_j$
touch when $j = i+1$ and do not touch otherwise.
 The sets $f_0$ and $f_n$ may be faces of $\Psi$, but need not be; in
this paper they are usually paths.
 If $R \subseteq \Si$ and $f_i \subseteq R$ for $1 \le i \le n-1$, we
say the face chain {\em goes through $R$\/}.

 A cyclic sequence of distinct faces $(f_1, f_2, \ldots, f_n)$
is called a {\em face ring of length $n$\/} if (i) $n = 2$, and $f_1$
and $f_2$ touch at least twice, or (ii) $n \ge 3$ and the sets $\bdy f_i
\cap \bdy f_j$, $i \ne j$, are pairwise disjoint, nonempty when $j =
i-1$ or $i+1$, and empty otherwise (subscripts interpreted modulo $n$). 
 We will use face rings only in closed $2$-cell embeddings, so we do not
need to consider situations in which a face `touches itself', i.e.,
face rings of length $1$.
 A face ring is {\em elementary\/} if (i) $n=2$ and the two faces touch
exactly twice, or (ii) $n \ge 3$ and any two faces touch at most
once.
 A face ring is {\em noncontractible\/} if $R = \bigcup_{i=0}^{n-1}
\clo{f_i}$ contains a noncontractible simple closed curve.

 The following observations will be useful for showing that face
rings are elementary.

 \begin{observation}\label{3conn}
 If $\Psi$ is an embedding of a $3$-connected graph, and two faces are
contained in some open disk, then the faces touch at most once.
 \end{observation}

 \begin{observation}\label{3conn3rep}
 If $\Psi$ is a $3$-representative embedding of a $3$-connected graph
then any two faces touch at most once.
 \end{observation}

 \section{Converting nonorientable surfaces to orientable surfaces}%
	\label{nonori2ori}

 It is not hard to turn an orientable embedding into a nonorientable 
embedding, by adding a crosscap in an arbitrary location.
 On the other hand, it is not easy in general to construct an orientable
embedding from an existing nonorientable embedding of a graph.
 The authors of \cite{FHRR,RT} developed some surgeries which
turn embeddings on the projective plane and on the Klein bottle into
orientable embeddings. However, the resulting orientable embeddings are
not necessarily closed $2$-cell embeddings.  In this section we develop
some surgeries to convert nonorientable embeddings into 
orientable embeddings.  We show that under certain conditions this can
be applied to convert a closed $2$-cell projective-planar embedding into
a closed $2$-cell orientable embedding.

 Our overall strategy will be as follows.  Suppose $G$ is embedded in a
nonorientable surface $\Si_1$, obtainable by inserting a set $\X$ of
crosscaps in an orientable surface $\Si_0$.  We insert an additional set
$\X'$ of crosscaps to get an embedding of $G$ on a nonorientable surface
$\Si_2$, in which all facial boundary component walks are cycles.  Then
we embed in $\Si_2$ a pseudograph $H$ disjoint from $G$, such that
cutting along $H$ destroys all the crosscaps in $\X \cup \X'$.  Capping
any holes due to cutting we obtain an embedding of $G$ on an orientable
surface $\Si_3$, in which all facial boundary component walks are still
cycles.  Removing any faces that are not open disks, and capping again,
we finish with a closed $2$-cell embedding of $G$ in an orientable
surface $\Si_4$.

 We first define some concepts related to crosscaps, next discuss
insertion of crosscaps, then examine cutting to remove nonorientability,
and finally apply these ideas to certain projective-planar embeddings.
 \smallskip

 We regard a {\em crosscap\/} in a $2$-manifold (with or without
boundary) as just a one-sided simple closed curve in the interior of the
manifold (rather than the usual definition, where a crosscap is a
neighborhood of such a curve, homeomorphic to a M\"obius strip).  To
{\em add a crosscap\/} to a $2$-manifold we remove a point or a set
homeomorphic to a closed disk not intersecting the boundary, locally
close the result by adding a boundary component homeomorphic to a
circle, then identify antipodal points of this circle to obtain a
one-sided simple closed curve.  If we removed a point $p$ or closed disk
$\De$, we call this {\em inserting a crosscap at $p$ or $\De$\/}.  If
$p$ is an interior point of an edge $e$ of an embedded graph, we call
this {\em inserting a crosscap on $e$\/}.  In figures we represent a
crosscap by a circle (representing the added boundary component)
with an X inside it.  To recover the original $2$-manifold (up to
homeomorphism) we may {\em collapse\/} the crosscap by identifying all
its points into a single point (the result is homeomorphic to what we
obtain by cutting out a M\"{o}bius-strip neighborhood of the crosscap
and capping the resulting hole with a disk).
 \smallskip

 We define the following surgeries for inserting crosscaps.  These, or
closely related operations, have been used in earlier papers such as
\cite{Z}, and generalize ideas used by Haggard \cite{Ha}.

\begin{operation}[Inserting crosscaps between faces or along a face
ring]\label{inscc}
 \leavevmode
 \begin{itemize}
 \item[(a)]%
 Let $f_1$ and $f_2$ be two distinct face occurrences at a vertex $v$.
(They may be different occurrences of the same face.)  Then $f_1$ and
$f_2$ partition the edges incident with $v$ into two intervals, $I_1$
and $I_2$.
 Choose a closed disk $\De$ close to $v$ such that all edges of $I_1$
cross it once, and no edges of $I_2$ intersect it.  Add a crosscap at
$\De$, re-embedding the parts of the edges of $I_1$ between $\De$ and
$v$ so that their order around $v$ is reversed.
 We call this {\/\em inserting a crosscap between $f_1$ and
$f_2$ near $v$\/}.  It does not matter whether we insert the crosscap
across $I_1$ or $I_2$: using the other one just amounts to pulling $v$
through the crosscap, which does not change the embedding.  If one of
$I_1$ or $I_2$ is a single edge, then this is equivalent to inserting a
crosscap on that edge.

 \item[(b)]%
 If $f_1$ and $f_2$ are distinct faces and $\bdy f_1 \cap \bdy f_2$ is a
path $P$ (possibly a single vertex) then the effect of inserting a
crosscap between $f_1$ and $f_2$ at any vertex of $P$, or on any edge of
$P$, is the same, so we just talk about {\/\em inserting a crosscap
between $f_1$ and $f_2$\/}.

 \item[(c)]%
 Suppose $\F = (f_1, f_2, \ldots, f_m)$ is an elementary face ring
of length $m \ge 3$.  Interpreting subscripts modulo $m$, insert a
crosscap between $f_i$ and $f_{i+1}$ for $1 \le i \le m$.  We call this
{\/\em inserting crosscaps along $\F$\/}.  (This can also be defined for
elementary face rings of length $2$, or for suitable face chains,
although we do not need this here.)

 \end{itemize}
 \end{operation}

 \begin{figure}
 \centerline{\epsfbox{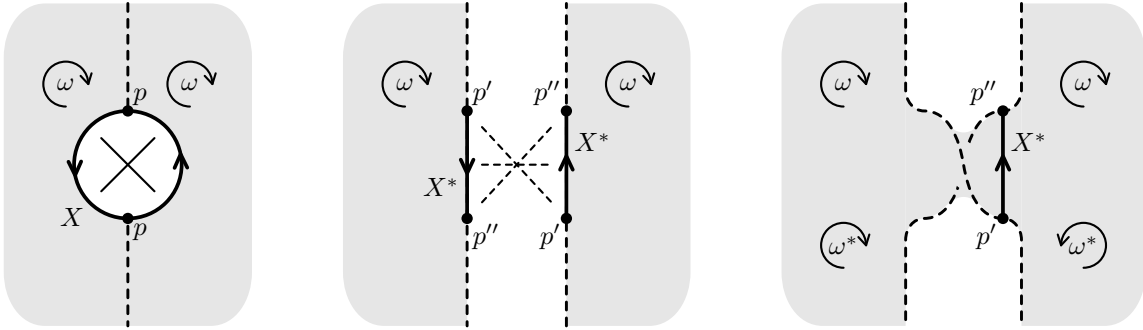}}
 \caption{Cutting a crosscap}\label{fig-inscc}
 \end{figure}

 While the above definition allows us some freedom in exactly how we
place the crosscaps, in practice it may be convenient to make a definite
choice about the location of the crosscaps, to help in tracing boundary
component walks in the new embedding.
 \smallskip

 Now we introduce our cutting operation.
 Our approach is fairly general, although we will need only a special case
of it for our projective-planar results.

 We first need to examine what happens if we cut through an individual
crosscap on a surface.  Suppose $\De$ is a neighborhood of a crosscap
$X$, which we think of as a circle $\cover{X}$ with antipodal points
identified.  There is a consistent local orientation $\ori$ on $\De -
X$.  Cut along a curve $\Ga$ which crosses $X$ at a single point $p$.
 If antipodal points of $\cover{X}$ were not identified, we would cut
$\De$ into two pieces $\Delta_1$ and $\Delta_2$.  Because we do identify
antipodal points, $\De_1$ and $\De_2$ remain joined along the cut
crosscap $X^*$, a line segment joining two copies of $p$.  $\De_1$ and
$\De_2$ are joined with a twist so we now have a local orientation
$\ori^*$ everywhere in the neighborhood of the cut crosscap $X^*$,
including on $X^*$ itself, which reverses relative to $\ori$ when we
cross $X^*$.  See Figure \ref{fig-inscc}: at left is a planned cut; in
the middle the cut crosscap is shown as two copies of the line segment
$X^*$ that are to be identified; and at right we see the result of the
identification.

  As mentioned earlier, we consider a nonorientable surface built up by
adding crosscaps to an orientable surface; we then cut through a
pseudograph $H$ intersecting those crosscaps to remove the
nonorientability.

 \begin{lemma}\label{OriSet1}
 Suppose $\Phi$ is a $2$-face-colorable embedding of a pseudograph $H$
on an orientable surface $\Si$.
 Suppose we choose a finite set of points $P$ so that each point of $P$
is an interior point of an edge of $H$.
 If we insert a crosscap $X_p$ at every $p \in P$, cut the resulting
surface along all edges of $H$, and use disks to cap the boundary
components of the resulting $2$-manifold with boundary, the result is a
finite union of pairwise disjoint orientable surfaces.
 \end{lemma}

 \begin{proof}
 It suffices to show that the $2$-manifold with boundary $\Si^*$
obtained by cutting along $H$ is orientable; then by capping we obtain a
union of orientable surfaces.

 Consider a fixed global orientation $\ori$ (local choice of clockwise
direction) for $\Si$.  When we add the crosscaps, $\ori$ gives a
consistent local orientation on a neighborhood of each crosscap $X_p$,
excluding $X_p$ itself.  Let $C_0$ and $C_1$ be the two color classes of
faces of $\Phi$.  Define an orientation $\ori^*$ on $\Si^*$ that is
equal to $\ori$ on each face in $C_0$ and opposite to $\ori$ on each
face in $C_1$.  When we pass between two faces of $\Phi$ in $\Si^*$, we
pass between a face in $C_0$ and a face in $C_1$ via a cut crosscap.  As
previously discussed, crossing the cut crosscap $X_p^*$ reverses the
orientation, in agreement with $\ori^*$.  Therefore $\ori^*$ provides a
consistent global orientation for $\Si^*$.
 \end{proof}

 In Lemma \ref{OriSet1}, $H$ need not be connected.  
 Also, we may allow $H$ to contain free loops, edges incident with no
vertex, which map to simple closed curves in an embedding; we can always
insert vertices to turn these into ordinary loops.

 Instead of starting with the pseudograph $H$ and inserting the
crosscaps, we may start with the crosscaps and add $H$.  Therefore, the
following is equivalent to Lemma \ref{OriSet1}.

 \begin{operation}\label{OriSet2}
 Suppose $\Si'$ is a nonorientable surface containing disjoint crosscaps
$X_1$, $X_2$, $\ldots$, $X_k$.  Suppose $\Phi'$ is an embedding on
$\Si'$ of a pseudograph $H$, such that each $X_i$ contains exactly one
point of $H$, an interior point of an edge that crosses $X_i$.  Suppose
that when we collapse every $X_i$, $1 \le i \le k$, we get a
$2$-face-colorable embedding of $H$ in an orientable surface $\Si$. 
Then if we cut $\Si'$ along all edges of $H$ and use disks to cap the
boundary components of the resulting $2$-manifold with boundary, the
result, $\Si''$, is a finite union of pairwise disjoint orientable
surfaces.
 \end{operation}

 In Operation \ref{OriSet2}, the condition that each crosscap $X_i$ be
crossed by exactly one edge of $H$, implying that the pseudograph after
collapsing every $X_i$ is still $H$, is important.
 If this is not the case, then the embedding obtained by cutting along
$H$ may not be orientable, even if collapsing the crosscaps yields a
$2$-face-colorable orientable embedding of some graph.  For example,
suppose we add four crosscaps $X_1, X_2, X_3, X_4$ to a sphere to get
$\Si'=\nos{4}$, and take $H$ to consist of two disjoint loops $\Ga_1$
and $\Ga_2$, where $\Ga_1$ and $\Ga_2$ both cross $X_1$ and $X_2$, $X_3$
is crossed only by $\Ga_1$ and $X_4$ is crossed only by $\Ga_2$.  When
we collapse the crosscaps we get a $2$-face-colorable embedding of a
graph on a sphere, but if we cut $\Si'$ along $\Ga_1$ and $\Ga_2$
the result is not orientable.

 With some extra conditions it is possible to allow more than one edge
of $H$ to cross a crosscap $X_i$, but we will not pursue the details here.
 \smallskip

 The pseudograph $H$ in Operation \ref{OriSet2} indicates where to cut. 
The graph we wish to embed is a different graph, $G$.  We assume that
$G$ also has an embedding $\Psi'$ in $\Si'$, disjoint from the embedding
$\Phi'$ of $H$.  Note that $G$ may cross some or all of the crosscaps
$X_i$ and still be disjoint from $H$; there may be several edges of $G$
crossing each crosscap, and an edge of $G$ may cross several crosscaps.
 In this context vertex-splitting in $H$ preserving
$2$-face-colorability allows us to assume, if we wish, that $H$ is a
union of vertex-disjoint cycles or free loops.

 If $G$ is connected, then when we cut along $H$ and cap, we get an
embedding $\Psi''$ of $G$ in an orientable surface, one of the connected
components of $\Si''$.
 We do not disturb the order of the edges around any of the vertices of
$G$, so $\Psi''$ has the same set of facial boundary component walks as
$\Psi'$, although the faces may not be open disks.
 \smallskip

 Now we combine Operations \ref{inscc} and \ref{OriSet2} to prove an
easy case of Theorem \ref{MainOSE}, which applies to graphs with
arbitrary vertex degrees, not just cubic graphs.

 \begin{theorem}\label{old4.2}
 Let $\Psi$ be a closed $2$-cell embedding of a $2$-connected graph $G$
in the projective plane.  Suppose $\Psi$ contains a noncontractible
elementary face ring $\F = (f_1, f_2, \ldots, f_l)$ which has odd
length $l \ge 3$.
 Then $G$ has a closed $2$-cell embedding in some orientable surface.
 \end{theorem}

 \begin{proof}
 %
 Represent the projective plane in the standard way as a disk with
antipodal boundary points identified; the boundary represents a crosscap
added to the sphere, which we call the {\em outer crosscap\/} $X_0$.  By
homotopic shifting we may ensure that the closure of exactly one face of
$\F$ intersects $X_0$, and $X_0$ cuts that face into exactly two
nonempty pieces.  See Figure \ref{fig-ofr}.

 \begin{figure}
 \centerline{\epsfbox{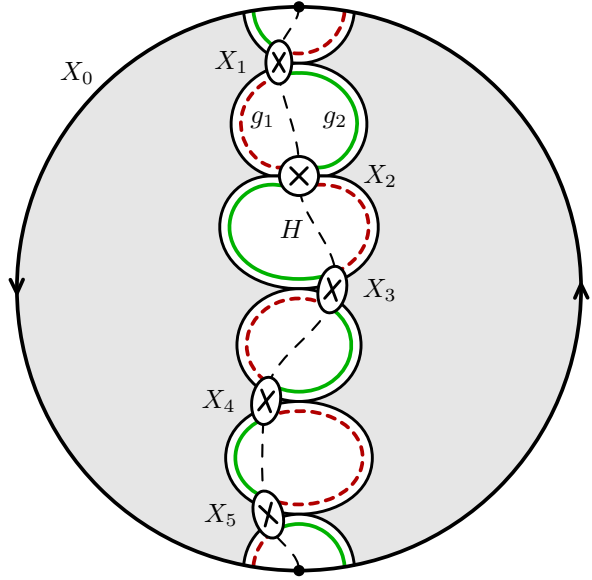}}
 \caption{Orienting using an elementary face ring of odd length}%
	\label{fig-ofr}
 \end{figure}

 Insert crosscaps $X_1$, $X_2$, $\ldots$, $X_l$ along $\F$ as in
Operation \ref{inscc}(c), to obtain a new embedding $\Psi'$ of $G$ in
$\nos{l+1}$.
 All faces of $\F$ turn into a new face $g$ in $\Psi'$ whose boundary
has two components $g_1$ and $g_2$.
 All other facial walks in $\Psi$ remain unchanged in $\Psi'$.  
 Because faces in $\F$ are disjoint except that consecutive faces
intersect in a single component, each of $g_1$ and $g_2$ is a cycle. 
Thus, all facial boundary component walks of $\Psi '$ are cycles. 

 By following $\F$, we may construct a simple closed curve $H$ that
passes through each face of $\F$ and each inserted crosscap exactly
once, and is disjoint from $G$.  $H$ crosses each of the $l+1$
crosscaps $X_0$, $X_1$, $\ldots$, $X_l$.
 If we consider $H$ as an embedded pseudograph with one vertex (any
point of $H - \bigcup_{i=0}^l X_i$) and one loop, then collapsing all
crosscaps leaves $H$ as a simple closed curve in the sphere, which is a
$2$-face-colorable embedding of $H$ in an orientable surface.
 Therefore, by Operation \ref{OriSet2}, cutting along $H$ and
capping yields an orientable surface, in which $G$ is embedded with all
facial boundary components cycles.  Removing any faces that are not open
disks and capping any resulting holes with disks, we obtain a closed
$2$-cell embedding of $G$ in an orientable surface, as required.
 \end{proof}

 Theorem \ref{old4.2} does not work with elementary face rings of even
length because then we get a face with a single boundary component that
self-intersects, instead of two components $g_1$ and $g_2$.

 In terms of our original strategy, here $\Si_0$ is the sphere, $\X =
\{X_0\}$, $\Si_1$ is the projective plane, $\X' = \{X_1, X_2, \ldots,
X_l\}$, $\Si_2$ is $\nos{l+1}$, and $\Si_3$ and $\Si_4$ are the
orientable surfaces in the last two sentences of the proof.

 The proof of Theorem \ref{old4.2} uses only a very simple version of
our strategy, and there is an alternative way to obtain the final
embedding in this proof, namely the surgery of Fiedler et al.
\cite[Lemma A]{FHRR}.
 However, we have examples of noncubic $2$-connected projective-planar
graphs where we can find orientable closed $2$-cell embeddings using
more complicated applications of Operations \ref{inscc} and
\ref{OriSet2}.
 Also, our approach continues the development of a coherent set of tools
for constructing closed $2$-cell embeddings, begun in papers such as
\cite{Z}.

 \section{Reductions}\label{Redn}

 In this section we describe reductions that allow us to restrict our
attention to a smaller class of projective-planar cubic
graphs.  Most of these reductions are standard, but the reduction we
give for $4$-cycles is new.  We state our results in a general setting
and then apply them to projective-planar cubic graphs.  We also prove a
technical result which will be used to deal with certain `planar
$4$-edge-cuts'.
 
 Let $\C_2$ be the class of $2$-connected cubic graphs, and let
$\G$ be a subclass of $\C_2$ closed under taking minors in $\C_2$ (i.e.,
if $G_1 \in \G$ and $G_2 \in \C_2$ is a minor of $G_1$, then $G_2 \in
\G$ also).  
 In constructing orientable closed $2$-cell embeddings for graphs in
$\G$, standard reductions used for cycle double covers (see \cite[pp.
4--5]{Ja}) can be applied to deal with any $2$-edge-cut, or nontrivial
$3$-edge-cut (a $3$-edge-cut that is not just the edges incident with a
single vertex).  Since we can exclude $2$-edge-cuts, we can also exclude
multiple edges.
 A $2$-connected cubic graph with no $2$-edge-cuts or nontrivial
$3$-edge-cuts is said to be {\em cyclically-$4$-edge-connected\/}, so we
may state the following.

 \begin{lemma}\label{minproperties}
 Given $\G$ as above, an element $G$ of $\G$ having no orientable closed
$2$-cell embedding (i.e., no orientable cycle double cover) and, subject
to that, fewest edges is simple and cyclically $4$-edge-connected.
 \end{lemma}

 Since $G$ has no nontrivial $3$-edge-cuts, and since $K_4$ has an
orientable embedding, it follows that $G$ as in Lemma
\ref{minproperties} has no triangles ($3$-cycles).
 For the Cycle Double Cover Conjecture there are also reductions to
exclude $4$-cycles and in fact cycles of length up to $11$, due to
Goddyn \cite{GoDiss}  (see also \cite[pp. 155--158]{Zhang}) and Huck
\cite{Hu}.
 However, these reductions do not work if we add the
condition of orientability.  Here we show that there is at least a
$4$-cycle reduction for the Orientable Cycle Double Cover Conjecture.

 \begin{lemma}\label{no4cycle}
 Given $\G$ as above, an element $G$ of $\G$ having no orientable closed
$2$-cell embedding (i.e., no orientable cycle double cover) and, subject
to that, fewest edges has no $4$-cycle.
 \end{lemma}

 \begin{proof}
 By Lemma \ref{minproperties} $G$ is simple,
cyclically-$4$-edge-connected, and has no triangles.  Since $K_{3,3}$
has a toroidal closed $2$-cell embedding (in which there are three
faces, all hamilton cycles), we know that $|V(G)| \ge 8$.

 Assume that $G$ contains a $4$-cycle $C = (v_1v_2v_3v_4)$.  Since $G$
has no triangles, each $v_i$ has a neighbour $u_i \notin V(C)$.
 Since $|V(G)| \ge 8$, cyclic-$4$-edge-connectivity implies that $u_1$,
$u_2$, $u_3$ and $u_4$ are distinct, and moreover that $G' = (G-V(C))
\cup \{u_1u_2, u_3u_4\}$ is $2$-connected.  Hence $G' \in \G$, so since
$G'$ has fewer edges than $G$ it has an orientable closed $2$-cell
embedding $\Psi'$.  Suppose the faces of $\Psi'$ containing $u_1u_2$ are
$f_1$ and $f_2$, where $\bdy f_1$ traversed clockwise uses directed edge
$u_2u_1$, and $\bdy f_2$ uses $u_1u_2$.  Suppose the faces containing
$u_3u_4$ are $g_1$ and $g_2$, where $\bdy g_1$ uses directed edge
$u_3u_4$ when traversed clockwise, and $\bdy g_2$ uses $u_4u_3$.  In
$\Psi'$ subdivide $u_1u_2$ to obtain the path $u_1v_1v_2u_2$, and
$u_3u_4$ to obtain the path $u_3v_3v_4u_4$.

 If $f_1=g_2$ we may add the edges $v_1v_4$ and $v_2v_3$ inside $f_1$ to
obtain an orientable closed $2$-cell embedding of $G$.  A similar
argument applies if $f_2=g_1$.  Thus, $f_1 \ne g_2$ and $f_2 \ne g_1$.

 If in addition $f_1 \ne g_1$ then we add a handle from $f_1$ to $g_1$
and add edges $v_1v_4$, $v_2v_3$ along the handle, to obtain an
orientable embedding of $G$.  In the new embedding the faces $f_1$ and
$g_1$ are replaced by $f_1'$ and $g_1'$, where $\bdy f_1'=(\bdy f_1 -
v_2v_1) \cup v_2v_3v_4v_1$ and $\bdy g_1'=(\bdy g_1 - v_3v_4) \cup
v_3v_2v_1v_4$.  Now $\bdy f_1'$ is a cycle because $f_1 \ne g_2$, and
$\bdy g_1'$ is a cycle because $g_1 \ne f_2$, so this is an orientable
closed $2$-cell embedding of $G$.
 Therefore, $f_1 = g_1$, and similarly $f_2=g_2$.

 Since $f_1=g_1$ and $f_2=g_2$ we may add the edge $v_1v_4$ inside $f_1$
and add the edge $v_2v_3$ inside $f_2$, to obtain an orientable closed
$2$-cell embedding of $G$.

 Since $G$ has an orientable closed $2$-cell embedding in all cases, our
assumption was false, and $G$ has no $4$-cycle, as required.
 \end{proof}

 We may apply our results when $\G$ is the class of projective-planar
cubic graphs.

 \begin{corollary}\label{minproperties2}
 Let $G$ be a $2$-connected projective-planar
cubic graph that has fewest edges subject to having no orientable
closed $2$-cell embedding.  Then $G$ is simple, cyclically-$4$-edge-connected,
and has no $3$- or $4$-cycles.
 \end{corollary}

 Now we introduce the technical result for dealing with certain `planar
$4$-edge-cuts.'  It is easier to prove it first in a dual form, for
near-triangulations.
 A {\em near-triangulation\/} is a graph embedded in the plane so that
every face is a triangle, except possibly for the outer face, which is a
cycle.
 A {\em separating cycle\/} in an embedded graph is a cycle whose
removal disconnects the surface into two components, each of which
contains at least one vertex of the graph.

 If $G$ is a connected graph embedded in the plane, $\bdy G$ represents
its outer walk.  An {\em interior\/} vertex of $G$ is a vertex not on
$\bdy G$.  An {\em interior path\/} in $G$ is a path none of whose
internal vertices lie on $\bdy G$ (although one or both ends may lie on
$\bdy G$).  If $C = (v_1 v_2 v_3 \ldots v_k)$ is a cycle of $G$, then
$I_G(C)$ or $I_G(v_1 v_2 v_3 \ldots v_k)$ represents the embedded
subgraph of $G$ on and inside $C$.  If $G$ is understood we just write
$I(C)$ or $I(v_1 v_2 v_3 \ldots v_k)$.

 $N_G(v)$ represents the set of neighbors of vertex $v$ in $G$, and
$N_G[v]$ represents $N_G(v) \cup \{v\}$.
 If $H$ is a subgraph of $G$, then a {\em chord\/} of $H$ is an edge
that is not in $H$ but whose two ends are in $H$.  If $x$ is a cutvertex
of $H$, then a chord of $H$ is {\em $x$-jumping\/} if its ends lie in
different components of $H - x$.

 \begin{observation}\label{A0}
 In a near-triangulation, a minimal cutset separating two given
nonadjacent vertices induces either a chordless separating cycle, or a
chordless interior path with ends on the outer cycle.
 \end{observation}

 \begin{observation}\label{C0}
 Suppose $G$ is a near-triangulation with outer
cycle $\partial G$ and $x$, $y$ are vertices of $\partial G$ with $xy
\notin E(\partial G)$.  Then either $G$ has an interior $xy$-path, or
$\bdy G - y$ has an $x$-jumping chord.
 \end{observation}

 \begin{proposition}\label{old5.4}
 Let $G$ be a simple graph (no loops or multiple edges) embedded in the
plane so that
 all faces are triangles except the outer face, which is a $4$-cycle
$(v_1v_2v_3v_4)$ in that clockwise order;
 there is at least one interior vertex; and
 there are no separating triangles.
 \begin{itemize}
 \item[\rm (i)]Then there is a chordless interior $v_1 v_3$-path in $G$.
 \item[\rm (ii)] Moreover, if all chordless interior $v_1v_3$-paths in
$G$ have length of the same parity (all even, or all odd), then $G$ has
an interior vertex of degree $4$.
 \end{itemize}
 \end{proposition}

 Nontrivial examples as in (ii) above do occur.  At left in Figure
\ref{fig-4c} is an example where all chordless interior $v_1 v_3$-paths
have even length; it has several interior vertices of degree $4$.

 \begin{proof}
 \startclaim
 For (i), $v_2 v_4 \notin E(G)$ because there are no separating
triangles, and hence by Observation \ref{C0} there is an interior $v_1
v_3$-path.  A shortest such path is chordless.

 \smallskip
 We prove (ii) by contradiction.  Assume it does not hold, and $G$
is a counterexample with fewest vertices.  So, all interior
$v_1v_3$-paths have length of the same parity, but there is no interior
vertex of degree $4$.
 If $|V(G)| = 5$, $G$ is a wheel with central vertex of degree $4$, so
we must have $|V(G)| \ge 6$.
 Since $G$ is a counterexample and there are no separating triangles,
every interior vertex has degree at least $5$.
 Since there are no separating triangles, $v_1v_3$, $v_2v_4 \notin
E(G)$.

 \begin{claim}\label{B0}
 If $C = (t_1 t_2 t_3 t_4)$ is a separating $4$-cycle then
there are chordless interior $t_1 t_3$-paths in $I(C)$ whose lengths
have different parities.
 \end{claim}

 \begin{claimproof}
 If not, then by minimality of $G$ there is an interior vertex
of $I(C)$ of degree $4$ in $I(C)$.  But this is also an interior vertex
of $G$ of degree $4$ in $G$, a contradiction.
 \end{claimproof}

 \begin{claim}\label{B1}
 There is no interior vertex adjacent to both $v_1$ and
$v_3$, or both $v_2$ and $v_4$.
 \end{claim}

 \begin{claimproof}
 Suppose there is an interior vertex $v$ with $vv_1, vv_3 \in E(G)$.
 Since $|V(G)| \ge 6$, at least one of the subgraphs $I(v_1 v_2
v_3 v)$ or $I(v_1 v v_3 v_4)$ has an interior vertex.  But this subgraph
contradicts Claim \ref{B0}.

 \begin{figure}
 \centerline{\epsfbox{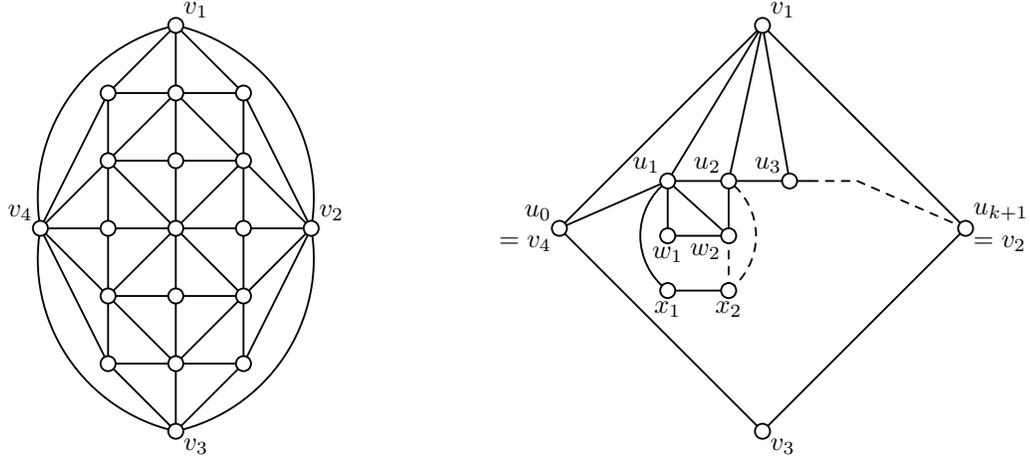}}
 \caption{Example with only even paths, and
 structure for the proof of Proposition \ref{old5.4}}%
 \label{fig-4c}
 \end{figure}

 If there is an interior vertex $v$ adjacent to both $v_2$ and $v_4$
then at least one of $G' = I(v_1 v_2 v v_4)$ and $G'' = I(v v_2 v_3
v_4)$ has an interior vertex; without loss of generality assume it is
$G'$.  There is a chordless interior $vv_3$-path $Q''$ in $G''$, by (i)
if $G''$ has an interior vertex, and because $v_2 v_4 \notin E(G)$ so
that $v v_3 \in E(G)$ if $G''$ has no interior vertex.  For every
chordless interior $v_1 v$-path $Q'$ in $G'$, $Q' \cup Q''$ is a
chordless interior $v_1 v_3$-path in $G$, so all such paths $Q'$ have
length of the same parity, contradicting Claim \ref{B0}.
 \end{claimproof}

 Let the neighbors of $v_1$ in anticlockwise order be $v_4 = u_0, u_1,
u_2, \ldots, u_k, u_{k+1} = v_2$.  By Claim \ref{B1}, $k \ge 2$.
 As all interior faces are triangles and there are no separating
triangles, $U = u_0u_1u_2 \ldots u_k u_{k+1}$ is a chordless path in
$G$.
 By Claim \ref{B1}, no vertex of $U - \{u_0, u_{k+1}\}$ is adjacent to
$v_3$.

 \begin{claim}\label{B2}
 Every $4$-cycle containing $v_1$ is $\bdy
G$, or has the form $(v_1 u_{i+2} u_{i+1} u_i)$, or has the form $(v_1
u_{i+1} x u_i)$ where $x$ is an interior vertex of $I(U \cup v_2v_3v_4)$.
 \end{claim}

 \begin{claimproof}
 Every $4$-cycle containing $v_1$ has the form $C= (v_1 u_j x u_i)$ with
$i < j$.  If $i = 0$ and $j = k+1$ then by Claim \ref{B1}, $C$ is $\bdy
G$.  So, without loss of generality, we may assume that $j \le k$.
 Then, also by Claim \ref{B1}, $x \ne v_3$.
 If $x$ is on $U$ then since $U$ is chordless, $C$ must be $(v_1 u_{i+2}
u_{i+1} u_i)$, as specified.  So we may assume that $x$ is an interior
vertex of $I(U \cup v_2v_3v_4)$.
 If $j = i+1$ then $C$ is as specified, so suppose $j \ge i+2$.  Let $G'
= I(C)$ and $G'' = I(u_0 u_1 \ldots u_i x u_j u_{j+1} \ldots u_{k+1}
v_3)$.  Since $U$ is chordless, Observation \ref{C0} implies that $G''$
has an interior $xv_3$-path.  Let $Q''$ be a shortest, hence chordless,
such path.  For any chordless interior $v_1x$-path $Q'$ in $G'$, $Q'
\cup Q''$ is a chordless interior $v_1 v_3$-path in $G$, so all such
paths $Q'$ have length of the same parity.  Since $G'$ has an interior
vertex, $u_{i+1}$, this contradicts Claim \ref{B0}.
 \end{claimproof}

 The following are immediate consequences of Claim \ref{B2}.

 \begin{claim}\label{B2A}
 There is no interior vertex $x$ not adjacent to $v_1$ but adjacent to
$u_i$ and $u_j$ with $|j-i| \ge 2$.
 \claimproofblock
 \end{claim}

 \begin{claim}\label{B2B}
 There is no separating $4$-cycle in $G$ of the form $(v_1 v_2 x y)$ or
$(v_1 x y v_4)$, or (by symmetry) $(v_3 v_2 x y)$ or $(v_3 x y v_4)$.
 \claimproofblock
 \end{claim}

 Since $u_1$ has degree at least $5$, there exist $w_1, w_2$ so that the
neighbors of $u_1$ in clockwise order are $v_4=u_0, v_1, u_2, w_2, w_1,
\ldots$.  Since $U$ is chordless, neither $w_1$ nor $w_2$ is on $U$, and
by Claim \ref{B1}, neither $w_1$ nor $w_2$ is $v_3$.  There are
triangular faces $(u_1 u_2 w_2)$ and $(u_1 w_2 w_1)$.  Since there are
no separating triangles, $u_0 w_2, w_1 u_2 \notin E(G)$.
 By Claim \ref{B2B}, $w_2 v_3 \notin E(G)$.  See the right side of
Figure~\ref{fig-4c}.

 Let $U' = u_0 u_1 w_1 w_2 u_2 u_3 \ldots u_{k+1}$, and let $H = I(U'
\cup v_2v_3v_4)$.  Since $U$ is chordless, since $u_0 w_2 \notin E(G)$,
and since $u_1u_2$ and $u_1 w_2$ are edges of $G$ but not $H$, $U'$ has
no $w_1$-jumping chord in $H$.

 \begin{claim}\label{w1v3path}
 There is an interior $w_1 v_3$-path in $H$ avoiding all neighbors (in
$G$) of $v_1$, $u_1$, $u_2$ and $w_2$, except $w_1$ itself.
 \end{claim}

 \begin{claimproof}
 Let $S = (V(U) \cup N_H[u_1] \cup N_H[u_2] \cup N_H[w_2]) - \{w_1\}$. 
Since $u_1v_3, w_2v_3, u_2v_3 \notin E(G)$, we know that $v_3 \notin S$.
 It suffices to show that there is a $w_1 v_3$-path in $H - S$.

 Let $S_1 = (N_H[u_1] - \{w_1\})$, and $S_2 = (N_H[w_2] -
\{w_1\}) \cup N_H[u_2] \cup \{u_4, u_5, \ldots, u_{k+1}\}$.  Then $S =
S_1 \cup S_2$.  Suppose $x \in S_1 \cap S_2$.  Since $U'$ has no
$w_1$-jumping chord in $H$, $x \notin V(U')$ and hence $x \in N_H(u_1)
\cap (N_H(w_2) \cup N_H(u_2))$.  But then either $(u_1 w_2 x)$ or $(u_1
u_2 x)$ is a separating triangle in $G$, a contradiction.  Hence $S_1
\cap S_2 = \emptyset$.

 Assume there is no $w_1v_3$-path in $H - S$. Then there is a minimal cutset
contained in $S$ separating $w_1$ and $v_3$ in $H$, which by Observation
\ref{A0} induces a chordless path $R$ starting on $u_0 u_1$ and ending
on $w_2 u_2 u_3 \ldots u_{k+1}$. 
 Since $U'$ has no $w_1$-jumping chord in $H$, there are no vertices of
$S_2$ on $u_0 u_1$, and no vertices of $S_1$ on $w_2 u_2 u_3 \ldots
u_{k+1}$.
 Hence, the first vertex of $R$ belongs to $S_1$ and the last to $S_2$. 
Let $x_1$ denote the last vertex of $R$ that belongs to $S_1$, and $x_2$
its immediate successor, which must belong to $S_2$.
 Since $U'$ has no $w_1$-jumping chord in $H$, we cannot have both $x_1,
x_2 \in V(U')$.

 Suppose $x_1 \in V(U')$, so $x_1 = u_0$ or $u_1$. Then $x_2 \in S_2 -
V(U') = (N_H(w_2) \cup N_H(u_2)) - V(U')$.  If $x_1 = u_1$ then $x_2 \in
S_1 \cap S_2$, a contradiction.  If $x_1 = u_0$ and $x_2 \in N(u_2)$
then $(v_1 u_2 x_2 u_0)$ contradicts Claim \ref{B2A}.  Thus, $x_1 = u_0$
and $x_2 \in N(w_2)$.
 Since $U'$ has no $w_1$-jumping chord in $H$, by Observation \ref{C0}
there is an interior $x_2v_3$-path in $I(u_0 x_2 w_2 u_2 u_3 \ldots
u_{k+1} v_3)$.
 Let $Q''$ be a shortest, hence chordless, such path.
 Since $I(u_0 u_1 w_2 x_2)$ has an interior vertex $w_1$, by Claim
\ref{B0} it has interior $u_1 x_2$-paths $Q_1'$ and $Q_2'$ whose lengths
have different parities.  But then $v_1 u_1 \cup Q_1' \cup Q''$ and $v_1
u_1 \cup Q_2' \cup Q''$ are chordless interior $v_1 v_3$-paths whose
lengths have different parities, a contradiction.


 Therefore, $x_1 \in S_1 - V(U') = N(u_1) - V(U')$.  We consider three
possibilities for $x_2$.  Cases (2) and (3) are shown at right in
Figure~\ref{fig-4c}.

 \smallskip\noindent
 (1) Suppose that $x_2 \in V(U')$, so that $x_2$ is a vertex of $w_2
u_2 u_3 \ldots u_{k+1}$.
 Since $x_1 u_1 \in E(G)$, we either get a separating triangle in $G$ if
$x_2 = w_2$ or $u_2$, or violate Claim \ref{B2A} if $x_2 = u_i$, $3 \le
i \le k+1$.

 \smallskip\noindent
 (2) Suppose that $x_2 \in N_H(w_2) - V(U')$.  Let $U'' = u_0 u_1 x_1
x_2 w_2 u_2 u_3 \ldots u_{k+1}$ and let $J = I(U'' \cup v_2v_3v_4)$.  If
$x_1$ is adjacent to a vertex $z$ of $w_2 u_2 u_3 \ldots u_{k+1}$ we get
a separating triangle if $z = w_2$ or $u_2$, or violate Claim \ref{B2A}
otherwise.  Also, $u_1 w_2, u_1 u_2 \notin E(J)$, and $U'$ has no
$w_1$-jumping chord in $H$.  Therefore, $U''$ has no $x_2$-jumping chord
in $J$.

 If there is an interior $x_2 v_3$-path in $J$ that avoids $N_J[u_1]$,
then we may take $Q''$ to be a shortest, hence chordless, such path. 
Then for any chordless interior $u_1 x_2$-path $Q'$ in $G' = I(u_1 w_2
x_2 x_1)$, $v_1 u_1 \cup Q' \cup Q''$ is a chordless interior $v_1
v_3$-path in $G$, so all such paths $Q'$ have length of the same parity.
 Since $G'$ has an interior vertex, $w_1$, this contradicts Claim \ref{B0}.

 Therefore, $S' = (V(U'') \cup N_J[u_1]) - \{x_2\}$ separates $x_2$ and
$v_3$ in $J$.  Let $S_1' = N_J[u_1]$ and $S_2' = \{w_2, u_2, u_3, \ldots,
u_{k+1}\}$, so that $S' = S_1' \cup S_2'$.  Because $U''$ has no
$x_2$-jumping chord in $J$, $S_1' \cap S_2' = \emptyset$.
 Applying Observation \ref{A0} to a minimal cutset contained in $S'$
separating $x_2$ and $v_3$, which induces a chordless path $R'$ which
starts on a vertex of $S_1' \cap V(U'')$ and ends on a vertex of $S_2'$,
we see that $R'$ has an edge $y_1y_2$ with $y_1 \in S_1'$, $y_2 \in
S_2'$.  Since $U''$ has no $x_2$-jumping chord in $J$, we cannot have
both $y_1, y_2 \in V(U'')$, so $y_1 \in N_J(u_1) - V(U'')$.  But then if
$y_2 = w_2$ or $u_2$ we have a separating triangle $(u_1 y_2 y_1)$, and
otherwise $y_1$ violates Claim \ref{B2A}.

 \smallskip\noindent
 (3) Suppose that $x_2 \in N_H(u_2) - V(U')$.  After modifying
$U''$ to be $u_0 u_1 x_1 x_2 u_2 u_3 \ldots u_{k+1}$, $G'$ to be $I(u_1
u_2 x_2 x_1)$, and $S_2'$ to be $\{u_2, u_3, \ldots, u_{k+1}\}$, the
proof is almost identical to (2) above.
 \smallskip

 So our assumption was incorrect, and there is a $w_1 v_3$-path in $H-S$.
 \end{claimproof}

 Let $Q$ be a $w_1v_3$-path as in Claim \ref{w1v3path} that is as short
as possible; then $Q$ is chordless.  As $Q$ contains no neighbor of
$v_1$, $u_1$, $u_2$ or $w_2$ except $w_1$, the paths $v_1 u_1 w_1 \cup
Q$ and $v_1 u_2 w_2 w_1 \cup Q$ are chordless interior $v_1 v_3$-paths
whose lengths have different parities, giving the final contradiction
that proves (ii).
 \end{proof} 

 \begin{corollary}\label{Z1}
 Let $G$ be a cyclically-$4$-edge-connected
cubic graph embedded in a surface, and let $C$ be a cycle of $G$ with a
fixed orientation bounding a closed disk $D$ in the surface.
 Suppose that precisely four edges touch $C$ from outside of $D$, at
vertices $u_1$, $u_2$, $u_3$, $u_4$ in that order around $C$, and let
$B_i = u_i C u_{i+1}$ (taking $u_5=u_1$).  If all $B_1B_3$-face chains
through $D$ have length of the same parity, then there is a $4$-cycle
face contained in $D$.
 \end{corollary}

 \begin{proof}
 Let $H$ be the intersection graph of the sets $B_1, B_2, B_3, B_4$ and
the closures of all faces inside $D$, and for each $i$ let $g_i$ denote
the face outside $D$ intersecting $D$ along $B_i$.
 Essentially $H$ is a subgraph of the dual of $G$, but we replace each
face $g_i$ by $B_i = \clo{g_i} \cap D$ to make sure that there are four
distinct vertices representing $g_1$, $g_2$, $g_3$ and $g_4$, even if
some of these faces are actually the same.  Since $G$ is cubic, $H$ is a
near-triangulation, with outer $4$-cycle $(B_1 B_2 B_3 B_4)$.  Since $G$
is cyclically-$4$-edge-connected, $H$ is simple and has no separating
triangles.  Apply Proposition \ref{old5.4} (ii) to $H$.
 \end{proof}

 \section{Main theorem}\label{Main}

 We are now ready to prove our main theorem.  The main step is a
structural result, Theorem \ref{Structure}.

 \begin{lemma}\label{ppfacering}
 Suppose $\Psi$ is a projective-planar embedding of a $3$-connected
cubic graph, with $\rho(\Psi) = m \ge 2$.  Then for any noncontractible
simple closed curve $\Ga$ with $|\Ga \cap G| = m$, the faces traversed
by $\Ga$, in their order along $\Ga$, form a noncontractible elementary
face ring.
 \end{lemma}

 \begin{proof}
 Label the faces along $\Ga$ as $f_1$, $f_2$, $\ldots$, $f_m$. 
Interpreting subscripts modulo $m$, clearly $f_i$ and $f_{i+1}$ touch
for each $i$.  If $f_i$ touches $f_j$ for some $j \ne i-1$, $i$ or
$i+1$, then we can find a noncontractible closed curve intersecting $G$
in fewer points than $\Ga$, a contradiction.  Thus, we have a
noncontractible face ring.

 If $m \ge 3$ then the face ring is elementary by Observation
\ref{3conn3rep}, so assume that $m = 2$.  
 If $\bdy f_1 \cap \bdy f_2$ has three or more components, then we can
find a contractible simple closed curve lying in $f_1 \cup f_2$ cutting
$G$ at exactly two vertices that form a cutset in $G$, contradicting
$3$-connectivity.  Thus, the face ring is elementary.
 \end{proof}

 \begin{theorem}\label{Structure}
 If $G$ is a cyclically-$4$-edge-connected cubic graph with a
$2$-representative embedding $\Psi$ in the projective plane, then the
embedding has a $4$-cycle face or a noncontractible elementary face ring
of odd length or both.
 \end{theorem}

 \begin{proof}
 \startclaim
 Assume for a contradiction that $G$ has neither a $4$-cycle face nor a
noncontractible elementary face ring of odd length.
 By Lemma \ref{ppfacering}, $\Psi$ has a noncontractible elementary face
ring of length $\rho(\Psi)$, so $\rho(\Psi) = 2n$ for some $n \ge 1$. 
Let this face ring be $\F = (f_1, f_2, \ldots, f_{2n})$.
Subscripts $i$ for $f_i$ are to be interpreted modulo $2n$.

 Since $G$ is cubic and $3$-connected, if $n=1$ then each component of
$\bdy f_1 \cap \bdy f_2$ is a single edge.  Similarly, if $n \ge 2$ then
each $\bdy f_i \cap \bdy f_{i+1}$ has one component which is a single
edge.
 Therefore, $F = \bigcup_{i=1}^{2n} \clo {f_i}$ is a closed M\"{o}bius
strip bounded by a cycle $L$.  $L$ contains distinct vertices $v_1, v_2,
\ldots, v_{4n}$ in that order such that $\bdy f_{i-1} \cap \bdy f_i$ is
the edge $v_i v_{2n+i}$.  Let $L_i = v_i L v_{i+1}$ (subscripts modulo
$4n$), so that the boundary of $f_i$ is $L_i \cup L_{2n+i} \cup \{v_i
v_{2n+1}, v_{i+1} v_{2n+i+1}\}$. Subscripts $i$ for $v_i$, $L_i$ and
related objects are to be interpreted modulo $4n$.
 Removing the interior of $F$ from the projective plane leaves a closed
disk $D$, which is the union of the closures of the faces not in $\F$,
and which contains all vertices of $G$.  We assume that $L$ goes around
$D$ clockwise, and all cycles contained in $D$ will also be oriented
clockwise.

 Suppose $\rho(\Psi) = 2$, so that $n = 1$.  Every $L_1 L_3$-face chain
through $D$ extends to an noncontractible elementary face ring using
Observation \ref{3conn}, and is therefore of even length.  But then by
Corollary \ref{Z1} there is a $4$-cycle face contained in $D$,
contradicting the fact that $G$ has no $4$-cycles.

 Thus, $\rho(\Psi) \ge 4$, so that $n \ge 2$.
 By Observation \ref{3conn3rep}, every face ring is elementary.  Thus,
every noncontractible face ring has even length.
 For each $i$, $1 \le i \le 4n$, let $\D_i$ be the set of faces in
$D$ whose closures intersect $L_i$, and $\D = \bigcup_{i=1}^{4n} \D_i$. 
 If $i, j \in \{1, 2, \ldots, 4n\}$, let $d(i,j)$ denote the distance
between $i$ and $j$ in the cyclic sequence $(1, 2, \ldots, 4n)$.

 We proceed by proving a sequence of claims.  Note that we implicitly
use the fact that $\rho \ge 4$.  In particular, all face rings we
contruct are valid if $\rho \ge 4$.

 \begin{claim}\label{D1}
 Since $G$ is cubic and $3$-connected and since $\rho
\ge 3$, for any two faces $f$ and $g$, $\bdy f \cap \bdy g$ is either
empty or a single edge.
 \claimproofblock
 \end{claim}

 \begin{claim}\label{D2}
 If $d(i, j) \ge 3$ then $\D_i \cap \D_j = \emptyset$.
 \end{claim}

 \begin{claimproof}
 Suppose not.  Without loss of generality we may assume that $i = 1$ and
$4 \le j \le 2n+1$.  Let $g \in \D_i \cap \D_j$.  We may construct a
noncontractible simple closed curve $\Ga$ passing through $f_1, g, f_j,
f_{j+1}, \ldots, f_{2n+1} = f_1$ and intersecting $G$ at $2n + 3 -j < 2n
= \rho$ points, a contradiction.
 \end{claimproof}

 \begin{claim}\label{D3}
 For every $d \in \D$, $\bdy d \cap L$ has exactly one component,
which is a subpath of $L$ with at least one edge.
 \end{claim}

 \begin{claimproof}
 Any component of $\bdy d \cap L$ cannot be a single vertex because
$G$ has no vertices of degree $4$ or more, so it is a subpath of $L$
with at least one edge.

 Suppose $\bdy d \cap L$ has two distinct components, $C_1$ and $C_2$. 
Suppose $C_1$ intersects $L_i$ and $C_2$ intersects $L_j$ where we
choose $i$ and $j$ so that $d(i,j)$ is as small as possible.  By Claim
\ref{D2}, $d(i,j) \le 2$.  If $i = j$ then $G$ has a $2$-edge-cut, and
if $d(i,j)=1$ then $G$ has a nontrivial $3$-edge-cut, contradicting the
fact that $G$ is cyclically-$4$-edge-connected.

 Suppose that $d(i,j) = 2$.  Without loss of generality assume that $i =
1$ and $j = 3$.  Using Claim \ref{D1}, suppose that $C_1 \cap L_1 = \bdy
d \cap \bdy f_1$ is the edge $x_1y_1$ where $x_1, y_1$ occur on that
order along $L_1$, and $C_2 \cap L_3 = \bdy d \cap \bdy f_3$ is the edge
$y_2x_2$ where $y_2, x_2$ occur in that order along $L_3$.  Since $G$ is
cubic, $x_1 \ne v_1$, $y_1 \ne v_2$, $y_2 \ne v_3$ and $x_2 \ne v_4$. 
Let $D'$ be the closed disk bounded by the cycle $y_1 (\bdy d) y_2 \cup
y_1 L y_2$.  Let $B_1 = y_1 L v_2$ and $B_3 = v_3 L y_2$.  Then for
every $B_1 B_3$-face chain $(B_1, g_1, g_2, \ldots, g_{k-1}, B_3)$ of
length $k$ through $D'$, we have a noncontractible face ring $(f_1, g_1,
g_2, \ldots, g_{k-1}, f_3, f_4, \ldots, f_{2n})$, which must have even
length, so that $k$ is always even.  Therefore, by Corollary \ref{Z1}
there is a $4$-cycle face contained in $D'$, contradicting the fact that
$G$ has no $4$-cycles.
 \end{claimproof}

 By Claim \ref{D3}, the elements of $\D$ can be cyclically ordered along
$L$ according to their intersection with $L$.  Within each $\D_i$ we may
linear order the elements of $\D_i$ along $L_i$ as $d_{i,1}, d_{i,2},
\ldots, d_{i,n_i}$, where $d_{i,1}$ contains $v_i$ and $d_{i,n_i}$
contains $v_{i+1}$.  By Claim \ref{D1}, $n_i$ is the length of $L_i$.
 Possibly $n_i=1$.  Note that $d_{i,n_i} = d_{i+1,1}$.  The following
is immediate.

 \begin{claim}\label{D3A}
 If $1 < j < n_i$ and $i \ne k$ then $d_{i,j}$ and $f_k$ do not touch.
 \claimproofblock
 \end{claim}

 \begin{claim}\label{D4}
 We do not have $n_i = n_{i+1} = 1$ for any $i$.
 \end{claim}

 \begin{claimproof}
 If $n_i = n_{i+1} = 1$ then $d_{i-1,n_{i-1}} = d_{i,1} = d_{i+1,1} =
d_{i+2,1}$, violating Claim \ref{D2}.
 \end{claimproof}

 \begin{claim}\label{D5A}
 Since $G$ has no nontrivial $3$-edge-cut, if $j \le k-2$ then
$d_{i,j}$ and $d_{i,k}$ do not touch.
 \claimproofblock
 \end{claim}

 \begin{claim}\label{D5B}
 If $j < n_i$ and $k > 1$ then $d_{i,j}$ and $d_{i+1,k}$ do not touch.
 \end{claim}

 \begin{claimproof}
 Without loss of generality assume that $i = 1$.  Suppose that $j <
n_1$, $k > 1$, and $d_{1,j}$ and $d_{2,k}$ touch.  Let $m$ be the
largest $m$ so that $d_{2,n_2} = d_{m,1}$; by Claim \ref{D4}, $m = 3$ or
$4$.  Let $\bdy d_{1,j} \cap L = x_1 L y_1$ and $\bdy d_{2,k} \cap L =
y_2 L x_2$.  Then $y_1$ and $y_2$ are internal vertices of $L_1$ and
$L_2$, respectively.  Let $y_3$ be the first vertex of $\bdy d_{2,k}$
encountered when travelling clockwise along $\bdy d_{1,j}$ from $y_1$.  Let
$D'$ be the closed disk bounded by the cycle $y_1 (\bdy d_{1,j}) y_3 \cup y_3
(\bdy d_{2,k}) y_2 \cup y_1Ly_2$.  Let $B_1 = y_1 L v_2$ and $B_3 = y_3 (\bdy
d_{2,k}) y_2$.
 Then for every $B_1 B_3$-face chain $(B_1, g_1, g_2, \ldots, g_{l-1},
B_3)$ of length $l$ through $D'$, by Claims \ref{D3A} and \ref{D5A} we
have a noncontractible face ring $(f_1, g_1, g_2, \ldots, g_{l-1},
d_{2,k}, d_{2,k+1}, \ldots, d_{2,n_2}=d_{m,1}, f_m, f_{m+1}, \ldots,
f_{2n})$, which must have even length, so that $l$ always has the same
parity.  Therefore, by Corollary \ref{Z1} there is a $4$-cycle face
contained in $D'$, contradicting the fact that $G$ has no $4$-cycles.
 \end{claimproof}

 \begin{claim}\label{D5C}
 If $d(i,j) \ge 4$ then no face in $\D_i$ touches a face in $\D_j$.
 \end{claim}

 \begin{claimproof}
 This is similar to the proof of Claim \ref{D2}.
 \end{claimproof}

 \begin{claim}\label{D6}
 If $n_{i-1} > 1$ and $n_{i+1} > 1$ then $n_i$ is odd.
 Equivalently, if $n_i$ is even then $n_{i-1} = 1$ or $n_{i+1} =1$.
 \end{claim}

 \begin{claimproof}
 Without loss of generality assume that $i = 2$.  If $n_1 > 1$ and $n_3
> 1$ then, by Claims \ref{D3A} and \ref{D5A}, $(f_1, d_{1,n_1} =
d_{2,1}, d_{2,2}, d_{2,3}, \allowbreak \ldots, \allowbreak d_{2,n_2} =
d_{3,1}, f_3, f_4, \ldots, f_{2n})$ is a noncontractible face ring,
which must have even length, so that $n_2$ must be odd.
 \end{claimproof}

 \begin{claim}\label{D7}
 If $n_{i-1} > 1$ and $n_{i+1} = 1$, or $n_{i-1} = 1$ and $n_{i+1}
> 1$, then $n_i$ is even.
 \end{claim}

 \begin{claimproof}
 Without loss of generality assume that $i=2$, where $n_1 > 1$ and $n_3
= 1$.  By Claim \ref{D4}, $n_4 > 1$.  Thus, and by Claims \ref{D3A} and
\ref{D5A}, $(f_1, d_{1,n_1} = d_{2,1}, d_{2,2}, d_{2,3}, \ldots,
d_{2,n_2}=d_{3,1}=d_{4,1}, f_4, f_5, \ldots, f_{2n})$ is a
noncontractible face ring, which must have even length, so that $n_2$
must be even.
 \end{claimproof}

 \begin{claim}\label{D8}
 If $n_{i-1} = n_{i+1} = 1$ then $n_i$ is odd and $n_i \ge 3$.
 \end{claim}

 \begin{claimproof}
 Without loss of generality assume that $i = 3$, and $n_2=n_4 =1$.  By
Claim \ref{D4}, $n_1 > 1$ and $n_5 > 1$.  Thus, and by Claims \ref{D3A}
and \ref{D5A}, $(f_1, d_{1,n_1} = d_{2,1} = d_{3,1}, d_{3,2}, d_{3,3},
\ldots, d_{3,n_3} = d_{4,1} = d_{5,1}, f_5, f_6, \ldots, f_{2n})$ is a
noncontractible face ring, which must have even length, so that $n_3$
must be odd.  By Claim \ref{D4}, $n_3 \ge 3$.
 \end{claimproof}

 \begin{claim}\label{D9}
 If $n_i$ is even then either $n_{i-1} = n_{i+2} = 1$ and
$n_{i+1}$ is even, or $n_{i+1} = n_{i-2} = 1$ and $n_{i-1}$ is even.
 \end{claim}

 \begin{claimproof}
 Without loss of generality assume that $i = 3$.  By Claim \ref{D6},
$n_2 = 1$ or $n_4 = 1$: again without loss of generality assume that
$n_2 = 1$.  By Claim \ref{D4}, $n_1 > 1$, and by Claim \ref{D8}, $n_4 >
1$.  If $n_5 > 1$, then $n_4$ is odd by Claim \ref{D6}, but then $(f_1,
d_{1,n_1}=d_{2,1}=d_{3,1}, d_{3,2}, d_{3,3}, \ldots, d_{3,n_3} =
d_{4,1}, d_{4,2}, d_{4,3}, \ldots, d_{4,n_4}=d_{5,1}, f_5, f_6, \ldots,
f_{2n})$ is a noncontractible face ring (by Claims \ref{D3A}, \ref{D5A}
and \ref{D5B}) of odd length, a contradiction.  Therefore $n_5 =1$, and
from Claim \ref{D7}, $n_4$ is even.
 \end{claimproof}

 \begin{claim}\label{D10}
 If $n_i = 1$ then either $n_{i+1}$ and $n_{i+2}$ are even
and $n_{i+3} = 1$, or $n_{i+1}$ is an odd number at least $3$ and
$n_{i+2}=1$.
 \end{claim}

 \begin{claimproof}
 Without loss of generality assume that $i=1$.  If $n_2$ is even then by
Claim \ref{D9} we know that $n_3$ is even and $n_4 = 1$.  If $n_2$ is
odd then by Claim \ref{D7} $n_3=1$.
 \end{claimproof}

 \begin{claim}\label{D11}
 We cannot have $n_i$, $n_{i+1}$, $n_{i+2}$, $n_{i+3}$ all greater than
$1$.
 \end{claim}

 \begin{claimproof}
 Without loss of generality assume that $i = 1$.  Suppose that $n_1,
n_2, n_3, n_4 > 1$.  By Claim \ref{D6}, $n_2$ and $n_3$ are odd.  By
Claims \ref{D3A}, \ref{D5A} and \ref{D5B}, $(f_1, d_{1,n_1}= d_{2,1},
d_{2,2}, d_{2,3}, \ldots, d_{2,n_2}=d_{3,1}, d_{3,2}, d_{3,3}, \ldots,
d_{3,n_3} = d_{4,1}, f_4, f_5, \ldots, f_{2n})$ is a noncontractible
face ring of odd length, a contradiction.
 \end{claimproof}

 Suppose now that no $n_i$ is even.  By Claim \ref{D11}, some $n_i$ is
equal to $1$, say $n_1 = 1$.  Since no $n_i$ is even, Claim \ref{D10}
implies that $n_2$, $n_4$, $n_6$, $\ldots$, $n_{4n}$ are all odd numbers
at least $3$, while $n_3 = n_5 = \ldots = n_{4n-1} = 1$.  In particular,
$n_1 = n_{2n+1} = 1$, which means that $f_1$ is a $4$-cycle face, a
contradiction.

 Therefore, some $n_i$ is even, and by Claim \ref{D9} either $n_{i-1} =
1$ or $n_{i+1} = 1$.  Without loss of generality we may assume that
$n_3$ is even and $n_2 = 1$.  By Claim \ref{D4}, $n_1 > 1$, and by Claim
\ref{D10}, $n_4$ is even and $n_5 = 1$, so that, by Claim \ref{D4}
again, $n_6 > 1$.  By Claims \ref{D3A}, \ref{D5A} and \ref{D5B}, $(f_1,
d_{1,n_1} = d_{2,1} = d_{3,1}, d_{3,2}, d_{3,3}, \ldots, d_{3,n_3} =
d_{4,1}, d_{4,2}, d_{4,3}, \ldots, d_{4,n_4} = d_{5,1} = d_{6,1}, f_6,
f_7, \ldots, f_{2n})$ is a noncontractible face ring of odd length, a
contradiction, if $\rho \ge 6$.

 Therefore, $\rho = 4$.  To satisfy Claims \ref{D10} and \ref{D11}, and
since some $n_i$ is even, we must have the following situation, or one
rotationally equivalent to it:
 $n_2 = 1$, $n_3$ even, $n_4$ even, $n_5 = 1$, $n_6$ even, $n_7$ even,
$n_8 = 1$, and $n_1$ odd and at least $3$.

 By Lemma \ref{ppfacering}, $(f_1, d_{1,n_1} = d_{2,1} = d_{3,1}, f_3,
f_4)$ is a noncontractible face ring.  The union of the closures of
these faces is a closed M\"obius strip with boundary cycle $L'$, which
we may divide into $8$ subpaths $L_i'$ in the same way that $L$ is
divided into subpaths $L_i$.  We see that $L_4, L_8 \subseteq L'$ and we
may orient $L'$ and label its subpaths so that $L_4' = L_4$, $L_8' =
L_8$, and the orientation of $L'$ agrees with that of $L$ on these two
subpaths.  If $n_i'$ is the length of $L_i'$ we see that $n_1' = n_1-1$
is even, $n_2'$ is unknown, $n_3' = n_3-1$ is odd, $n_4' = n_4$ is even,
$n_5'=n_5+1=2$, $n_6'=1$, $n_7'=n_7+1$ is odd, and $n_8'=n_8=1$.  By
Claim \ref{D9}, $n_3' = 1$ so that $n_3 = 2$.  By symmetry, $n_7 = 2$
also.

 In the same way, $(f_4, d_{4,n_4}=d_{5,1}=d_{6,1}, f_6, f_7)$ is a
noncontractible face ring, and we get a closed M\"obius strip whose
boundary cycle is divided into $8$ subpaths of lengths $n_i''$, $1 \le i
\le 8$, where we see that $n_1'' = 1$, $n_2'' = n_2+1 = 2$, $n_3'' = n_3
= 2$, $n_4'' = n_4-1$ is odd, $n_5''$ is unknown, $n_6'' = n_6-1$ is
odd, $n_7''=n_7 =2$, and $n_8''=n_8+1 = 2$.  By Claim \ref{D9}, $n_4'' =
n_6'' = 1$ so that $n_4 = n_6 = 2$.

  Now all even numbers $n_i$ ($n_3$, $n_4$, $n_6$, $n_7$) have been
shown to equal $2$.  By the same reasoning, all even numbers $n_i'$ must
be $2$.  In particular, $n_1' = n_1-1 = 2$, so that $n_1 = 3$.


 Since $n_1 = 3$ and $n_3=n_4=n_6=n_7=2$ we may write $L = (v_1 w_{1,1}
w_{1,2} v_2 v_3 w_3 v_4 w_4 v_5 v_6 w_6 v_7 \ab w_7 v_8)$.  For each $u
\in \{w_{1,1}, w_{1,2}, w_3, w_4, w_6, w_7\}$, let $u'$ denote the
neighbor of $u$ to which $u$ is joined by an edge that is not an edge of
$L$.  By Claim \ref{D3}, no such $u'$ is a vertex of $L$.

 If $d_{3,2}$ does not touch $d_{1,2}$ then  $(f_4=f_8, d_{8,1} =
d_{1,1}, d_{1,2}, d_{1,3} = d_{2,1} = d_{3,1}, d_{3,2}=d_{4,1})$ is a
noncontractible face ring of length $5$, which is odd, a contradiction. 
So $d_{3,2}$ touches $d_{1,2}$.

 Write $\bdy d_{1,2} \cap d_{3,2} = x_2y_2$ where $x_2$ precedes $y_2$
in the clockwise order around $\bdy d_{3,2}$.  By Claim \ref{D3}, $x_2$
and $y_2$ are not vertices of $L$.  If $w_{1,2}'$, $w_3'$ and $y_2$ are
not all equal, then $\{w_{1,2} w_{1,2}', w_3 w_3', x_2 y_2\}$ is a
nontrivial $3$-edge-cut, a contradiction.
 Therefore, $y_2 = w_{1,2}' = w_3'$ is a vertex shared by $d_{1,2}$,
$d_{1,3} = d_{2,1} = d_{3,1}$, and $d_{3,2}$.  Similarly, there is a
vertex $y_1 = w_{1,1}' = w_7'$ shared by $d_{1,2}$, $d_{1,1} = d_{8,1} =
d_{7,2}$, and $d_{7,1}$, with neighbor $x_1 \ne w_{1,1}, w_7$, where
neither $x_1$ nor $y_1$ is on $L$.
 Since $d_{1,2}$ is not a $3$- or $4$-cycle face, $y_1 \ne y_2$ and $y_1
y_2 \notin E(G)$, so that $x_1 \ne y_2$ and $x_2 \ne y_1$.

 Since $d_{5,1}$ is not a $4$-cycle face, $w_4 w_6 \notin E(G)$.  Since
the neighbors of $y_1$ are $x_1$, $w_7$ and $w_{1,1}$, $y_1$ is adjacent
to neither $w_4$ nor $w_6$; similarly, $y_2$ is adjacent to neither
$w_4$ nor $w_6$.  Thus, $\{y_1, y_2, w_4, w_6\}$ is an independent set
and $x_1, x_2, w_4', w_6'$ do not belong to this set.

 Now there is a contractible simple closed curve intersecting the
embedding at precisely four points, one interior point of each of the
four edges  $y_1 x_1$, $y_2 x_2$, $w_4 w_4'$, $w_6 w_6'$, in that order.
 This curve bounds an open disk $\De$ containing $x_1$, $x_2$, $w_4'$
and $w_6'$.  Let $H$ be the subgraph of $G$ induced by $V(G) \cap \De$. 
Since $G$ is cyclically-$4$-edge-connected, $H$ is either a single edge,
or is a $2$-connected graph embedded in a closed disk $D'$ bounded by a
cycle $C'$ through distinct vertices $x_1$, $x_2$, $w_4'$, $w_6'$ in
that order (and possibly containing other vertices).
 If $H$ is a single edge, then either
 $x_1 = x_2$, $w_4' = w_6'$, $x_1 w_4' \in E(G)$, and there is a
noncontractible simple closed curve through $f_3$, $d_{3,2}$ and
$d_{7,1}$ that intersects $G$ at only $3$ points;
 or $x_1 = w_6'$, $x_2 = w_4'$, $x_1 w_4' \in E(G)$, and there is a
noncontractible simple closed curve through $f_1$, $d_{1,2}$ and
$d_{5,1}$ that intersects $G$ at only $3$ points.
 In either case we have a contradiction, so $H$ is $2$-connected and we
have $D'$ and $C'$ as described above.
 Let $B_1 = x_1 C' x_2 = x_2 (\bdy d_{1,2}) x_1$ and $B_3 = w_4' C' w_6'
= w_6' (\bdy d_{5,1}) w_4'$.  Then for every $B_1 B_3$-face chain $(B_1,
g_1, g_2, \ldots, g_{l-1}, B_3)$ of length $l$ through $D'$, $(f_1,
d_{1,2}, g_1, g_2, \ldots, g_{l-1}, f_{5,1})$ is a noncontractible face
ring, which must have even length, so that $l$ is always even. 
Therefore, by Corollary \ref{Z1} there is a $4$-cycle face contained in
$D'$, a contradiction.

 Since every possibility leads to a contradiction, our original
assumption must be wrong, and $\Psi$ does have either a $4$-cycle face
or a noncontractible elementary face ring of odd length.
 \end{proof}

 Now we prove the main result, which we restate.

 \begin{theoremMainOSE}
 \MainOSEtext
 \end{theoremMainOSE}

 \begin{proof}
 Suppose $G$ is a $2$-connected cubic graph with a projective-planar
embedding $\Psi$ but with no orientable closed $2$-cell embedding and
that, subject to these conditions, $G$ has a minimum number of vertices.
 By Corollary \ref{minproperties2}, $G$ is simple,
cyclically-$4$-edge-connected (hence $3$-connected), and has no $3$- or
$4$-cycles.

 If $\rho(\Psi) \le 1$ then $G$ can be embedded in the plane
\cite{N,RV}, and hence has a spherical closed $2$-cell embedding.
 Otherwise, by Theorem \ref{Structure}, since $G$ has no $4$-cycle,
$\Psi$ must have a noncontractible elementary face ring of odd length. 
 Then by Theorem \ref{old4.2}, $G$ has an orientable closed $2$-cell
embedding, as required.
 \end{proof}

 We would like to strengthen Theorem \ref{MainOSE} to say that we have a
$k$-face-colorable orientable closed $2$-cell embedding for some fixed
small $k$.  Then we could improve Corollary \ref{MainOCDC} to say that we
have an orientable $k$-cycle double cover (the cycles can be colored
using at most $k$ colors so that each edge is contained in two cycles of
different colors).  However, this seems difficult.  The obstacle is our
$4$-cycle reduction: we see no obvious way to avoid increasing the
number of face colors needed when $f_1 \ne g_2$ and $f_2 \ne g_1$ in the
proof of Lemma \ref{no4cycle}.

 \section*{Acknowledgment}

 The authors thank John Ratcliffe for useful discussions regarding the
surgery operation discussed in Lemma \ref{OriSet1}.

 \end{document}